\documentclass[12pt]{article}

\usepackage[american]{babel}
\usepackage{amsmath}
\usepackage{latexsym}
\usepackage{amssymb}
\usepackage{theorem}
\usepackage[small,nohug]{diagrams}
\usepackage{mathrsfs}
\usepackage{calligra}
\DeclareMathAlphabet{\mathcalligra}{T1}{calligra}{m}{n}
\DeclareFontShape{T1}{calligra}{m}{n}{<->s*[1.2]callig15}{}
\theoremstyle{change}
\newtheorem{Thm}{Theorem}[section]
\newtheorem{Cor}[Thm]{Corollary}

\newtheorem{Lem}[Thm]{Lemma}
{\theorembodyfont{\rmfamily}
\newtheorem{Num}[Thm]{}

}

\newcommand{\pr}{\mathrm{pr}}
\renewcommand{\phi}{\varphi}
\renewcommand{\rho}{\varrho}

\newcommand{\bra}[1]{\langle#1\rangle}
\newcommand{\proof}{\par\medskip\rm\emph{Proof. }}

\newcommand{\qed}{\ \hglue 0pt plus 1filll $\Box$}

\newcommand{\mapstoo}{\longmapsto}

\newcommand{\RR}{\mathbb{R}}

\renewcommand{\SS}{\mathbb{S}}

\newcommand{\QQ}{\mathbb{Q}}
\newcommand{\CC}{\mathbb{C}}

\newcommand{\id}{\mathrm{id}}
\newcommand{\SKIP}[1]{}

\renewcommand{\emptyset}{\varnothing}

\newcommand{\Aut}{\mathrm{Aut}}

\renewcommand{\setminus}{-}

\newcommand{\incl}{\mathrm {incl}}
\newcommand{\defi}{\buildrel\rm def\over=}
\def\.{{\cdot}}
\begin{document}

%

\def\L{\mathfrak L}
\def\I{\Bbb I}
\def\R{\Bbb R}
\def\Z{\Bbb Z}
\def\incl{{\mathrm incl}}
\def\defi{\buildrel\rm def\over=}
\def\.{{\cdot}}

\title{{\bf Transitive actions of locally compact groups\\ on locally contractible
spaces}}

\author{Karl H.~Hofmann and Linus Kramer}
\date{}
\maketitle

\begin{abstract}\noindent
Suppose that $X=G/K$ is the quotient of a locally compact
group by a closed subgroup. If $X$ is locally contractible
and connected, we prove that $X$ is a manifold. If the $G$-action 
is faithful, then $G$ is a Lie group.
\end{abstract}

\section{Introduction}

In 1974 J. Szenthe stated the following result in \cite{janos}.
{\em Let a $\sigma$-compact locally compact group $G$, with
compact quotient $G/G^\circ$, act as a transitive and faithful
transformation group on a locally contractible space $X$.
Then $X$ is a manifold and $G$ is a Lie group.} 
This result, which may be viewed as a
solution of Hilbert's 5th problem for transformation groups,
has been widely used since then. However, it was discovered
in 2011 that Szenthe's proof contains a serious gap. In
the present paper we close this gap, proving Szenthe's
statement in a different way. 
Independently and simultaneously, this result was 
also proved by A.A.~George Michael \cite{adeliii} and
by S.~Antonyan and T.~Dobrowolski \cite{AD}.
The last section of our
paper contains some more comments on the history of this
problem.
Our main results are as follows.

\medskip
\noindent
\textbf{Theorem A}
{\em Let $G$ be a compact group and let $K\subseteq G$ be a closed subgroup.
Suppose that the homogeneous space $X=G/K$ contains a nonempty open 
subset $V\subseteq X$ which is contractible in $X$. Then $X$ is a closed manifold.
If $N=\bigcap\{gKg^{-1}\mid g\in G\}$ denotes the kernel of the $G$-action,
then $G/N$ is a compact Lie group acting transitively on $X$.}

\medskip
\noindent
For the case of a locally compact group, we need a stronger
topological assumption on the coset space $X$.

\medskip
\noindent
\textbf{Theorem B}
{\em Let $G$ be a locally compact group and let $K\subseteq G$ be a closed subgroup.
Suppose that the homogeneous space $X=G/K$ is locally contractible.
Then $X$ is a manifold. If $X$ is connected or if $G/G^\circ$ is compact,
and if  $N=\bigcap\{gKg^{-1}\mid g\in G\}$ denotes the kernel of the action,
then $G/N$ is a Lie group acting transitively on $X$.}

\medskip
\noindent
In the course of the proof, we need the following extension of Iwasawa's
Local Splitting Theorem, which may be interesting in its own right.
A local version of this result was proved by Glu\v skov~\cite{Gluskov}.
This result is also proved in \cite{hofmori} Theorem~4.1 in a different way.

\medskip
\noindent
\textbf{Theorem C}
{\em Let $G$ be a locally compact group and let $U\subseteq G$ be a 
neighborhood of the identity. Then there exist a compact
subgroup $N\subseteq U$, a simply connected Lie group $L$ 
and an open homomorphism $\phi:N\times L\rTo G$ with discrete kernel, such that
$\phi(n,1)=n$ for all $n\in N$.}

\medskip\noindent
A variation of the theme of this article appears in the third edition
of \cite{HMCompact} (2013) in Sections 10.72 to 10.93.

\medskip
\noindent\textbf{Conventions and Terminology}
All maps and group homomorphisms are assumed to be continuous and
all spaces and groups are assumed to be Hausdorff unless stated
otherwise. Topological countability assumptions will be stated
explicitly whenever they are used. By a Lie group we mean a
locally compact group $G$ which is a smooth manifold, such that 
multiplication and inversion are smooth maps, without any further
topological countability assumptions.

The identity component of a topological group $G$
is denoted by $G^\circ$. This is always a closed normal subgroup.
We denote by $\mathcal N(G)$ the collection of all closed normal subgroups $N\unlhd G$ with
the property that $G/N$ is a Lie group, and we note that $G\in\mathcal N(G)$.

For subgroups $P,Q\subseteq G$ we put
\begin{align*}
 \mathrm{Cen}_P(Q)=&\{p\in P\mid pq=qp\text{ for all }q\in Q\} \\
 \mathrm{Cen}(Q)=&\mathrm{Cen}_Q(Q)                            \\
 \mathrm{Nor}_P(Q)=&\{p\in P\mid pQp^{-1}=Q\}                  \\
 [P,Q]=&\bra{[p,q]\mid p\in P\text{ and }q\in Q}
\end{align*}
The unit interval is denoted by $[0,1]=\{t\in\RR\mid 0\leq t\leq 1\}$.
For a map $h:X\times [0,1]\rTo Y$ we write $h_t(x)=h(x,t)$.
The projection map of a cartesian product $X\times Y$ is
denoted by $\pr_X:(x,y)\mapstoo x$, and similarly
$\pr_Y:(x,y)\mapstoo y$.

\medskip\noindent\textbf{Acknowledgment}
%
We thank S.~Antonyan, S.~Morris, and the referee for their helpful comments.
Also, we are grateful to R.~McCallum for his careful reading of the manuscript.

\section{Some homotopy theory of compact groups}
We begin by collecting some results which we shall need in the proof
of Theorem~A. 
Unless stated otherwise, homotopies are not required to preserve
base points.
Recall that a map $E\rTo B$ is called a \emph{fibration} if it
has the \emph{homotopy lifting property} for every space $X$.
This means that for every commutative diagram
\begin{diagram}[width=4em]
X\times\{0\} & \rTo & E\\
\dInto& \ruDotsto^{\tilde h}&\dTo\\
X\times[0,1]&\rTo^h &B,
\end{diagram}
the dotted lift $\tilde h$ exists. The following result is related
to the notion of \emph{irreducibility} in \cite{Dugundji} p.~394
and in \cite{Madison}.
\begin{Lem}
\label{HomotopyLemma}
Let $E$ be a space with the property that every homotopy equivalence
$E\rTo^\simeq E$ is surjective. Suppose that $p:E\rTo B$ is a
surjective fibration. Then also every homotopy equivalence $B\rTo^\simeq B$
is surjective. If $p$ is homotopic to a map $p':E\rTo B$, then
$p'$ is also surjective.

\proof
First we note the following. If $\xi:B\rTo B$ is a homotopy equivalence
with homotopy inverse $\eta$, then $\xi\circ\eta$ is, by definition, homotopic
to $\id_B$. In order to show the surjectivity of such a map $\xi$, 
it suffices therefore to prove the surjectivity of every map which is
homotopic to the identity $\id_B$.

Suppose that $h:B\times[0,1]\rTo B$ is a map with
$h_0=\id_B$. Then the map $h':E\times[0,1]\rTo B$
with $h'_t(x)=h_t(p(x))$ is a homotopy between
$p=h'_0$ and $p'=h'_1$. We now show that $p'$ is surjective.
Since $E$ is a fibration, there exists a
lift $\tilde h':E\times[0,1]\rTo E$ of $h'$,
with $\tilde h'_0=\id_E$.
\begin{diagram}[width=4em]
&& E\\
&\ruDotsto^{\tilde h'}& \dTo\\
E\times{[0,1]} & \rTo^{h'} & B  
\end{diagram}
By our assumptions on $E$, the map $\tilde h_1':E\rTo E$ is
surjective. It follows that $B=p(\tilde h_1'(E))=p'(E)=h_1(B)$.
\qed
\end{Lem}
We now recall several results about the structure of compact groups.
\begin{Thm}[Approximation by compact Lie groups]
\label{Approximation}
Let $G$ be a compact group. Then every neighborhood $V$ of the identity
contains a closed normal subgroup $N\unlhd G$ such that
$G/N$ is a compact Lie group. The set $\mathcal N(G)$ consisting 
of all closed normal subgroups $N\unlhd G$ such that $G/N$ is a Lie
group is a filter basis converging to the identity.

\proof
See Theorem 9.1 and 2.43 in \cite{HMCompact}.
\qed
\end{Thm}
\begin{Lem}
\label{NSSLemma}
Let $G$ be a compact group and $N\unlhd G$ a closed normal subgroup.
If $N$ and $G/N$ are Lie groups, then $G$ is a Lie group as well.

\proof
We show that $G$ has no small subgroups, see 
\cite{HMCompact} Theorem~2.40. 
Let $W\subseteq G/N$ be
a neighborhood of the identity which contains no nontrivial subgroup
and let $V$ be its preimage in $G$ under the projection
$G\rTo G/N$.
Let $U\subseteq G$ be a neighborhood of the identity such that
$U\cap N$ does not contain a nontrivial subgroup of $N$. Then
$U\cap V$ contains no nontrivial subgroup of $G$.
\qed
\end{Lem}
\begin{Thm}[Complementation of normal subgroups]
\label{Complementation}
Suppose that $G$ is a compact group and that $N\unlhd G$ is a closed normal
subgroup. If $G/N$ is connected, then there exists a closed connected subgroup
$M\unlhd G^\circ$ with $G=MN$ and $[M,N]=\{1\}$, and such that $M\cap N$ is totally
disconnected,

\proof
See Theorem 9.77 of \cite{HMCompact}. 
\qed
\end{Thm}
The following fact about abstract permutation groups is well-known, 
see \cite{DM} Theorem~4.2A. In order to make this article
self-contained we include a proof. We remark that Lemma~\ref{CentralizerLemma}
below indicates that in a topological setting we have here an interesting
example of forced continuity of abstract group actions.
\begin{Num}\label{PermutationGroupFact}
Suppose that $X$ is a set and
that $x\in X$. Let $\mathrm{Sym}(X)$ denote the group of all 
permutations of $X$. Suppose that $M\subseteq\mathrm{Sym}(X)$ 
is a transitive subgroup, with $x$-stabilizer $H=M_x$. 
Then we may identify $X$ with the quotient
$M/H$ in an $M$-equivariant way, thereby identifying $x$ with $H\in M/H$. 
Let $N=\mathrm{Nor}_M(H)\subseteq M$.
Then $N$ acts (from the left) on $M/H$ via
$(n,mH)\mapstoo mHn^{-1}=mn^{-1}H$. In this way we obtain
a homomorphism $\alpha$
\[
 1\rTo H\rInto N\rTo^\alpha \mathrm{Sym}(X)
\]
with $\alpha(n)(mH)=mn^{-1}H$. The kernel of $\alpha$ is $H$. 
Obviously, the factor group $N/H$ centralizes in this action the group $M$. 

We claim that $\alpha$ maps $N$ onto the centralizer 
$C=\mathrm{Cen}_{\mathrm{Sym}(X)}(M)$.
Suppose that $c\in C$. By transitivity of $M$, there exists $n\in M$ with $n^{-1}(x)=c(x)$.
For $h\in H$ we have then
$hn^{-1}(x)=hc(x)=ch(x)=c(x)=n^{-1}(x)$, whence $n\in N$.
For $m\in M$ we have
$cm(x)=mc(x)=mn^{-1}(x)$. The right-hand side is precisely the
$N$-action defined above.
\end{Num}
We need, however, a topological version of this fact
which is somewhat stronger than \cite{Onishchik} p.~73.
\begin{Lem}
\label{CentralizerLemma}
Suppose that $G$ is a compact group, that $K\subseteq G$ is a closed subgroup
such that $K$ contains no 
nontrivial normal subgroup of $G$. Suppose that $M\subseteq G$ is a closed 
subgroup such that $G=KM$. Then the compact group $\mathrm{Cen}_G(M)$
injects continuously into the compact group
$\mathrm{Nor}_M(M\cap K)/M\cap K$.

\proof
Let $N=\mathrm{Nor}_M(M\cap K)$ and $H=M\cap K$.
By \ref{PermutationGroupFact}
we may view $\mathrm{Cen}_G(M)$ as a subgroup of the abstract group
$C=\mathrm{Cen}_{\mathrm{Sym}(G/K)}(M)$.
By \ref{PermutationGroupFact} we have an injective abstract group
homomorphism $\beta:\mathrm{Cen}_G(M)\rTo N/H$.
The graph
\[
\beta=\{(c,nH)\in \mathrm{Cen}_G(M)\times N/H\mid nc\in K\}
\]
of $\beta$ is closed and thus $\beta$ is continuous, 
see for example \cite{Dugundji} XI.2.7.
\qed
\end{Lem}
\begin{Thm}
\label{AlmostLieGroup}
Let $M$ be a compact connected group. If there exists a closed
totally disconnected
normal subgroup $D\subseteq M$ such that $M/D$ is a Lie group, then
there exist simple simply connected compact Lie groups
$S_1,\ldots, S_r$ and a compact connected finite dimensional abelian
group $A$ and a central surjective homomorphism 
\[
A\times S_1\times \cdots\times S_r\rTo M
\]
with a totally disconnected kernel.

\proof By \cite{HMCompact} Proposition~9.47, the Lie algebra of $M$ is
isomorphic to the Lie algebra of $M/D$ and hence
has a finite dimension. By \cite{HMCompact} Theorem~9.52,
$M$ has the properties claimed above. 
\qed
\end{Thm}
The following Theorems~\ref{ThmMM1} and \ref{ThmMM2} are due to Madison and
Mostert~\cite{Madison}. 
In order to make the paper self-contained, we include proofs.
\begin{Thm}[Madison-Mostert]
\label{ThmMM1}
Let $G$ be a compact group and let $P,Q\subseteq G$ be closed subgroups, with
$P\subseteq Q$. Then the natural map
\[
 G/P\rTo G/Q
\]
is a fibration.

\proof
Suppose we are given a commutative diagram
\begin{diagram}[width=4em]
X\times\{0\} & \rTo & G/P\\
\dInto& \ruDotsto &\dTo\\
X\times[0,1]&\rTo &G/Q.
\end{diagram}
We have to show the existence of the dotted map. To this end we consider the poset
$\mathcal P$ consisting of pairs $(\phi,N)$, where $N\unlhd G$ is a closed normal subgroup
and $\phi:X\times[0,1]\rTo G/PN$ is a map that fits into the commutative diagram
\begin{diagram}[width=4em]
X\times\{0\} & \rTo & G/P & \rTo & G/PN\\
\dInto&&\dTo& \ruDotsto(4,2)&\dTo \\
X\times[0,1]&\rTo &G/Q& \rTo& G/QN.
\end{diagram}
We put $(\phi,N)\geq(\psi,M)$ if $N\subseteq M$ and if the diagram
\begin{diagram}[width=4em]
&& G/P & \rTo & G/PN & \rTo & G/PM\\
&&\dTo& \ruDotsto(4,2) &\dTo& \ruDotsto(6,2) &\dTo\\
X\times[0,1]&\rTo &G/Q& \rTo& G/QN&\rTo&G/QM.
\end{diagram}
commutes. We claim that this partial order is inductive.
Suppose that $\mathcal T\subseteq\mathcal P$ is a linearly ordered subset.
Let $L=\bigcap\{N\mid (N,\phi)\in\mathcal T\}\unlhd G$. We have natural maps
\begin{diagram}[height=3em,width=5em,nohug]
&&\prod_{(N,\phi)\in\mathcal T}G/PN&\lTo^\alpha &G/PL\\
&\ruTo^\gamma &\dTo&\ruDotsto(4,2)&\dTo\\
 X \times[0,1] & \rTo^\delta &\prod_{(N,\phi)\in\mathcal T}G/QN&\lTo^\beta&G/QL
\end{diagram}
Now $\alpha$ and $\beta$ are injective and hence homeomorphisms onto their respective images. 
Moreover, the image of $\gamma$ is contained in the image of $\alpha$ and the image of
$\delta$ is contained in the image of $\beta$. Thus we can fill in the dotted map and
obtain an upper bound $(L,\psi)$ of $\mathcal T$.
By Zorn's Lemma, $\mathcal P$ has maximal elements. 

Suppose that $N\unlhd G$ is a compact normal subgroup and that $G/N$ is a Lie group.
Therefore the canonical map $G/PN\rTo G/QN$ is a locally trivial fiber bundle, see
\cite{Warner}, Theorem~3.58,
and hence a fibration, see \cite{Dugundji}, Theorem 4.2. 
Thus there exists a map $\phi:X\times[0,1]\rTo G/PN$
such that $(N,\phi)$ is contained in $\mathcal P$.
By Theorem~\ref{Approximation}, 
there exist arbitrarily small compact normal subgroups 
$N\unlhd G$ such that $G/N$ is a Lie group.
It follows that the maximal elements in $\mathcal P$ are of the form $(\{1\},\phi)$, 
and these elements solve the initial lifting problem.
\qed
\end{Thm}
A more general result than Theorem~\ref{ThmMM1} can be found in \cite{Sklj} Theorem~15.
As a commentary we mention the following corollary (which will not be used here).
\begin{Cor}
Let $G$ be a compact group and let $P\subseteq Q\subseteq G$ be closed subgroups. 
Then there is a long exact sequence for the homotopy groups
\[
 \rTo\pi_k(Q/P)\rTo\pi_k(G/P)\rTo\pi_k(G/Q)\rTo\relax {.}
\]
\qed
\end{Cor}
Similarly, the Leray-Serre Spectral Sequence may be applied to such a fibration
(because every fibration is a Serre fibration).
For the next lemma we remark that \v Cech cohomology and Alexander-Spanier
cohomology agree for compact spaces. 
See also \cite{HMos} App.~II~3.15 and App.~III~2.11.
\begin{Lem}
\label{Injectivity}
Suppose that a compact totally disconnected group $D$ acts on a compact space $X$.
Then the orbit space map $p:X\rTo D\backslash X$ induces an injection
in \v Cech cohomology with rational coefficients,
\[
\check H^*(X;\QQ)\lTo^{p^*} \check H^*(D\backslash X;\QQ)\lTo0.
\]

\proof
We have by Theorem~\ref{Approximation} that
\[
 D=\varprojlim\{D/E\mid E\in\mathcal{N}(D)\}
\]
and the groups $D/E$ are finite (in other words, $D$ is a profinite group).
It follows that
\[
 X=\varprojlim\{E\backslash X\mid E\in\mathcal{N}(D)\}.
\]
For $E,F\in\mathcal N(D)$ with $F\subseteq E$ we have that $E/F$ is finite and
thus
\[
 \check H^*(F\backslash X;\QQ)\lTo\check H^*(E\backslash X;\QQ)
\]
is injective, see~\cite{BredonCompact} Theorem III.7.2.
Projective limits of compact spaces commute with \v Cech cohomology,
see~\cite{ES} X.3. The claim follows now.
\qed
\end{Lem}
The following result is also proved in \cite{HofmannMislove} Corollary 1.9.
\begin{Thm}[Madison-Mostert]
\label{ThmMM2}
Let $G$ be a compact group, let $P\subseteq G$ be a closed subgroup, and let
$\xi: G/P\rTo G/P$ be a homotopy equivalence. Then $\xi$ is surjective.
\end{Thm}
\begin{proof}
As we noted in the proof of Lemma~\ref{HomotopyLemma}, we have to show that
for every map $h:G/P\times[0,1]\rTo G/P$ with $h_0=\id_{G/P}$, the map
$h_1$ is surjective.
Let $\mathcal C$ denote the class of all pairs $(G,P)$ of compact
groups where this holds. Our goal is to show that
$\mathcal C$ is the class of all compact group pairs.

\medskip\noindent\emph{Claim 1. If $G$ is a compact connected group,
then $(G,1)$ is in $\mathcal C$.}
Suppose that this is false. Then there exists a map
$h:G\times[0,1]\rTo G$ with $h_0=\id_G$, and $h_1:G\rTo G$ is not surjective.
By Theorem~\ref{Approximation}, there exists a closed normal
subgroup $N\unlhd G$ such that $G/N$ is a compact connected 
Lie group, say of dimension $r$,
and an element $g\in G$ such that $gN\cap h_1(G)=\emptyset$. 
By Theorem~\ref{Complementation},
there exists a compact connected subgroup $M\subseteq G$ such that
$G=MN$, and $D=M\cap N$ is totally disconnected. Then $L=M/D\cong G/N$
is a compact connected $r$-dimensional Lie group.
We note that $r>0$, since otherwise we would have $G=N$.
Then we have a commutative diagram 
\begin{diagram}[width=4em]
M&\rInto&G & \lInto^j & h_1(G) &\lTo^{h_1}&G\\
\dTo&&\dTo &&\dTo \\
L&\rTo^\cong&G/N &\lTo &G/N\setminus\{gN\}.
\end{diagram}
Because of the homotopy $h_1\simeq h_0=\id_G$, the restriction map 
$\check H^*(G;\QQ)\rTo^{j^*} \check H^*(h_1(G);\QQ)$ is injective.
We have therefore in $r$-dimensional cohomology a commutative diagram
\begin{diagram}[width=5em]
\check H^r(M;\QQ)&\lTo &\check H^r(G;\QQ) & \rTo^{j^*} & \check H^r(h_1(G);\QQ) \\
\uTo^{(1)}&&\uTo^{(2)} &&\uTo \\
\check H^r(L;\QQ)&\lTo^\cong&\check H^r(G/N;\QQ) &\rTo 
&\check H^r(G/N\setminus\{gN\};\QQ).
\end{diagram}
The map (1) is injective by Lemma~\ref{Injectivity}. Thus (2) is also injective.
The map $j^*$ is injective by the previous remark. 
The compact connected Lie group $L$ is $\QQ$-orientable, whence
$\check H^r(L;\QQ)\cong\QQ$. On the other hand, 
$\check H^r(G/N\setminus\{gN\};\QQ)=0$. This is a contradiction.

\medskip\noindent\emph{Claim 2. Suppose that $G$ is a compact group.
Then $(G,1)$ is in $\mathcal C$.}
Suppose that $a\in G$ is not in the image of $h_1:G\rTo G$. 
Then  $h'_t(g)=a^{-1}h_t(ag)$ is a map from the identity component $G^\circ$
of $G$ to itself with $h'_0=\id_{G^\circ}$, and $1$ is not in the image of
$h_1'$, contradicting Claim~1.

\medskip\noindent\emph{Claim 3. If $(G,1)$ is in $\mathcal C$ and if $P\subseteq G$
is closed, then $(G,P)$ is in $\mathcal C$.}
This follows from Claim~2 and Lemma~\ref{HomotopyLemma}, since $G\rTo G/P$ is a fibration by
Theorem~\ref{ThmMM1}.
\qed
\end{proof}

\section{The proof of Theorem A}
\label{ThmASection}
Let $V$ be a subset of a topological space $X$. We say that $V$
is \emph{contractible in $X$} if there is a map
$f:V\times[0,1]\rTo X$, $(v,t)\mapstoo f_t(v)$,
such that $f_0(v)=v$ for all $v\in V$ and
$f_1$ is a constant map from $V$ into $X$.  In the
case that $V$ is open,  we note that the image of 
$f$ in $X$ is a path-connected set with nonempty interior.
We call a topological  space $X$ \emph{piecewise contractible} 
if it satisfies the following condition.
\begin{description}
 \item[(PC)] There is a cover of $X$ by nonempty open subsets 
 which are contractible in $X$.
\end{description}
If $X$ admits a transitive group of homeomorphisms, then
(PC) is obviously equivalent to the condition that there is
some nonempty open set $V\subseteq X$ which is contractible in $X$.
\begin{Lem}
\label{ProductPC}
A product space $X\times Y$ is piecewise contractible of and only if
$X$ and $Y$ are piecewise contractible.

\proof
If $V\subseteq X$ and $W\subseteq Y$ are open subsets and if
$f:V\times[0,1]\rTo X$ and $g:W\times[0,1]\rTo Y$ contract
$V$ and $W$ in $X$ and $Y$, respectively,
then the map $h:V\times W\times[0,1]\rTo X\times Y$, 
with $h_t(v,w)=(f_t(v),g_t(w))$
contracts $V\times W$ in $X\times Y$. In this way we obtain
the required open cover of $X\times Y$.

Conversely, 
if $U\subseteq X\times Y$ is an open subset containing
the point $(x,y)$, and if $h:U\times[0,1]\rTo X\times Y$
contracts $U$ in $X\times Y$, then there is 
an open neighborhood $V\subseteq X$ of $x$ such that $V\times\{y\}\subseteq U$.
Then $f_t(v)=\pr_X(h_t(v,y))$ contracts $V$ in $X$.
\qed
\end{Lem}
We now prove three preparatory lemmas in order to obtain Theorem~A.
\begin{Lem}
\label{Lemma0}
Let $G$ be a compact group and let $K\subseteq G$ be a closed subgroup.
If $X=G/K$ is piecewise contractible, then the subgroup $G^\circ K$ is open
in $G$, the quotient space $G^\circ K/K\cong G^\circ/G^\circ\cap K$
is piecewise contractible, and there is a $G^\circ$-equivariant homeomorphism
\[
G/K\cong (G^\circ K/K)\times D
\]
for some finite set $D$.

\proof
The space $G/G^\circ K$ is compact and 
totally disconnected, see \cite{HMCompact} Prop.~10.32 and the following remark. 
Thus the free right action of the compact group $G^\circ K$
on $G$ has a totally disconnected orbit space.
Now the result \cite{HMCompact} Theorem~10.35 
on the Existence of Global Cross Sections applies and shows that
$G$ is homeomorphic to $G^\circ K \times D$ with a totally disconnected
compact space $D$ (homeomorphic to $G/G^\circ K$)
in such a fashion that the action  of $G^\circ K$ is
by multiplication on the first factor. In other words, we have
a $G^\circ K$-equivariant homeomorphism $G\cong G^\circ K\times (G/G^\circ K)$,
and therefore a $G^\circ$-equivariant homeomorphism
\[
 G/K\cong (G^\circ K/K)\times (G/G^\circ K).
\]
Now $G/K$ is piecewise contractible. Hence by Lemma~\ref{ProductPC}, 
the totally disconnected compact homogeneous space $D\cong G/G^\circ K$
is piecewise contractible. This implies that each point of $D$ is open, and hence
$D$ is finite. So $G^\circ K$ is open in $G$ and $G^\circ K/K$ is open in $G/K$.
Lemma~\ref{ProductPC} implies also that $G^\circ K/K$ is piecewise contractible.
\qed  
\end{Lem}
\begin{Lem}
\label{Lemma1}
Let $G$ be a compact connected group and let $K\subseteq G$ be a closed subgroup.
If $X=G/K$ is piecewise contractible, then there exists a closed normal connected subgroup
$M\unlhd G$ with $G=KM$ (i.e. $M$ acts transitively on $X$), and
$M$ has a closed central totally disconnected subgroup $D\unlhd M$
such that $M/D$ is a compact connected Lie group.

\proof
Let $V\subseteq G/K$ be an open neighborhood of the coset $K\in G/K$ which is
contractible in $X$ by a map $f:V\times[0,1]\rTo G/K$. 
By Theorem~\ref{Approximation} there exists a closed normal subgroup $N\in\mathcal N(G)$ 
with $NK/K\subseteq V$.
Let $M\unlhd G$ be a complement of $N$ as in Theorem~\ref{Complementation}.
Since $G=MN$, the group $N$ acts transitively on $G/KM$, and we
have by Theorem~\ref{ThmMM1} a surjective fibration 
\[
NK/K\cong N/K\cap N\rTo^p N/(KM)\cap N\cong G/KM.
\]
We define $f':(N/K\cap N)\times[0,1]\rTo G/KM$
by $f'_t(n(K\cap N))=f_t(nK)KM$.
Then $f_0'=p$ and $f_1'$ is constant. It follows from Lemma~\ref{HomotopyLemma}
that $G=KM$. Let $D=M\cap N$.
Then $M/D\cong G/N$ is a compact connected Lie group. 
Since $[M,N]=\{1\}$, we have that $D$ is central in $M$.
\qed
\end{Lem}
The proof of the following lemma is partially adapted from \cite{HoNe} Prop.~3.5.
\begin{Lem}
\label{Lemma2}
Let $G,K,X,M$ be as in Lemma~\ref{Lemma1}. If $G$ acts faithfully
on $X$, then $M$ is a compact Lie group.

\proof We put $L=M\cap K$ and we identify $X$ with the quotient space $M/L$.
It follows from Theorem~\ref{AlmostLieGroup}
that there is a compact connected semisimple Lie group
$S$, a compact connected finite dimensional abelian group $A$, and a surjective homomorphism
$q:A\times S\rTo M$ with a totally disconnected kernel $E$.
Passing to a quotient, we can also assume that $E$
intersects the factors $\{1\}\times S$ and $A\times\{1\}$ trivially. 
We want to show that $A$ is a Lie group.

Since $M$ acts faithfully and transitively on $X$, the central subgroup $q(A\times\{1\})$ intersects
the stabilizer $L$ trivially and hence $A$ acts freely on $M/L$.
There is an open neighborhood $V\subseteq X$ of $L\in M/L$
and a map 
\[
f:V\times [0,1]\rTo X
\]
that contracts $V$ in $X$. By Theorem~\ref{Approximation},
there is a closed subgroup $B\subseteq A$ such that 
$A/B\cong \mathbb T^m$ is a compact torus of finite dimension $m$, and such
that the $B$-orbit of $L\in M/L$ is contained in $V$. 
From the action of $A\times S$ on $M/L$ we have by
Theorem~\ref{ThmMM1} a fibration 
\[
 A\times S\rTo M/L
\]
We note that $B$ acts freely on $M/L$.
Hence we may identify $B$ with the $B$-orbit of $L\in M/L$. Then $f$ gives us a map
$g:B\times[0,1]\rTo M/L$, with $g_0(b)=bL$ and $g_1$ constant.
We now lift this to 
$\tilde g:B\times[0,1]\rTo A\times S$,
such that $\tilde g_0(b)=b$ for all $b\in B$. We define
$h:B\times[0,1]\rTo A$ by $h_t(b)=\pr_A(\tilde g_t(b))$.
Now there is a covering homomorphism $\phi:B\times\RR^m\rTo A$ with discrete
kernel, with $\phi(b,0)=b$ for all $b\in B$, see \cite{HMPro} Thm.~13.17 and 13.20.
We lift $h$ to a map $\tilde h:B\times[0,1]\rTo B\times\RR^m$.
The composite
\[
 B\times [0,1]\rTo^{\tilde h}B\times\RR^m\rTo^{\pr_B} B
\]
is a homotopy between $\id_B$ and a constant map. Thus $B$ is contractible
and hence by \ref{ThmMM2} trivial.
\qed
\end{Lem}
Now we have collected all ingredients for the proof of Theorem A.
\begin{Thm}
\label{MainTheoremADetailed}
Suppose that $G$ is a compact group and that $K$ is a closed subgroup.
Assume also that the action of $G$ on the homogeneous space $X=G/K$ is
faithful, i.e. that $K$ contains no nontrivial normal subgroup of $G$.
If $X$ is piecewise contractible, then $G$ is a compact Lie group and
$X$ is a closed, but not necessarily connected manifold.

\proof
By Lemma~\ref{Lemma0} we have a $G^\circ$-equivariant homeomorphism 
$G/K\cong (G^\circ/G^\circ\cap K)\times D$, for some finite set $D$,
and $G^\circ/G^\circ\cap K$ is piecewise contractible.
By Lemma~\ref{Lemma2} there exists a closed normal connected Lie group
$M\unlhd G^\circ$ acting transitively on $G^\circ/G^\circ\cap K$.
The group $G^\circ$ decomposes by Theorem~\ref{Complementation} 
as $G^\circ=NM$, where $N\unlhd G$
is a closed normal subgroup that centralizes $M$. Since $G^\circ$ acts
faithfully and transitively
on $G^\circ/G^\circ\cap K$, it follows from Lemma~\ref{CentralizerLemma}
that $N$ is isomorphic to a closed subgroup of the Lie group
$\mathrm{Nor}_M(M\cap K)/M\cap K$.
By Lemma~\ref{NSSLemma}, the group $G^\circ$ is a compact connected Lie group. 
Now we want to show the same for the group $G^\circ K$. 
The group $\mathrm{Cen}_G(G^\circ)$ is normal in $G$ and has therefore
trivial intersection with $K$, because the action is faithful.
Thus $\mathrm{Cen}_K(G^\circ)=\{1\}$. Therefore $K$ injects into
the compact Lie group $\mathrm{Aut}(G^\circ)$. 
Now both $G^\circ$ and $G^\circ K/G^\circ\cong K/K\cap G^\circ$
are Lie groups and thus $G^\circ K$ is also a Lie group by
Lemma~\ref{NSSLemma}.
Finally, $G^\circ K$ is open in $G$, hence $G$ 
is also a Lie group.
\qed
\end{Thm}
\begin{Cor}
\label{CorTheoremADetailed}
Let $G$ be a compact group and $K$ a closed subgroup. If $X=G/K$ is
piecewise contractible, then $G/K$ is a closed, but not necessarily connected
manifold. The quotient $G/N$, where $N=\bigcap\{gKg^{-1}\mid g\in G\}$ is
a compact Lie group that acts faithfully and transitively on $X$.
\qed
\end{Cor}
Theorem~\ref{MainTheoremADetailed} and Corollary~\ref{CorTheoremADetailed}
yield Theorem~A in the introduction.
For the sake of completeness, we restate the result in terms of transformation groups.
\begin{Cor}
Let $X$ be a compact locally contractible space. Suppose that a compact group $G$
acts as a transitive transformation group on $X$, via a continuous map
\[
 G\times X\rTo X.
\]
Then $X$ is a closed manifold. If the $G$-action is faithful, then $G$ is a compact Lie group.

\proof
Let $K$ denote the stabilizer of a point $x\in X$.
Since $G$ is compact, the natural continuous map $G/K\rTo X$ is a homeomorphism.
\qed
\end{Cor}

\section{A splitting result for locally compact groups}

The following result was proved by Iwasawa in \cite{iwa}, p.~547, Theorem~11.
\begin{Thm}[Iwasawa's Local Splitting Theorem]
\label{IwasawaSplittingTheorem} 
\begin{sloppypar}Let $G$ be a locally compact connected group.
Then $G$ has arbitrarily small neighborhoods which are
of the form $NC$ such that $N$ is a compact normal subgroup and
$C$ is an open n-cell which is a 
local Lie group commuting elementwise with $N$, such
that $(n,c)\mapstoo nc$ is a homeomorphism $N\times C\rTo NC$.
\end{sloppypar}
\qed
\end{Thm}
Iwasawa assumes in \emph{loc.cit.} that $G$ is a projective limit of Lie groups. 
However, in the process of settling  Hilbert's Fifth Problem (see 
\cite{montzip}, p.~184), Yamabe showed that every locally compact group
has an open subgroup which is a projective limit of Lie groups 
(see \cite{montzip}, p.~175). 
We now extend Iwasawa's Splitting Theorem 
to not necessarily connected locally compact groups. 
This result is essentially Glu\v skov's Theorem~A in \cite{Gluskov}.
It is also proved in \cite{hofmori}~Theorem~4.1 in a different way.
We begin with two lemmas. 
\begin{Lem}
\label{LemmaL1}
Let $A$ and $B$ compact connected abelian groups. Suppose that we are given 
continuous homomorphisms $\RR^m\rTo^\phi B\lTo^p A$. If $p$ is surjective, then the 
lifting problem
\begin{diagram}[width=4em]
&&A\\
&\ruDotsto^{\tilde\phi}&\dTo_p\\
\RR^m&\rTo^\phi& B
\end{diagram}
has a solution $\tilde\phi$.

\proof
We dualize the diagram. The Pontrjagin duals $\widehat A$ and $\widehat B$
are discrete torsion free abelian groups and $\widehat p$ is injective. 
Moreover, $\widehat{\RR^m}\cong\RR^m$. The dual problem
\begin{diagram}[width=4em]
&&\widehat A\\
&\ldDotsto^{\widehat{\tilde\phi}}&\uTo_{\widehat p}\\
\widehat{\RR^m}&\lTo^{\widehat\phi}&\widehat B
\end{diagram}
clearly has a solution $\widehat{\tilde\phi}$ (for example by passing to the divisible hulls
of $\widehat A$ and $\widehat B$, which are $\QQ$-vector spaces). 
Note that we do not have to worry about continuity, since
both $\widehat A$ and $\widehat B$ are discrete groups. Now we dualize the solution
$\widehat{\tilde\phi}$ of this problem.
\qed
\end{Lem}
\begin{Lem}
\label{LemmaL2}
Let $L$ be a simply connected Lie group and let $N$ be a compact group.
Let $\alpha:L\rTo\Aut(N)$ be a homomorphism, and consider the 
semidirect product $N\rtimes_\alpha L$. If $L$ centralizes under this map the
identity component $N^\circ$, then there is an isomorphism
\[
 \phi:N\rtimes_\alpha L\rTo^\cong N\times L
\]
which restricts to the identity on $N\times 1$.

\proof The $\alpha$-image of $L$ is contained in the identity component
$\Aut(N)^\circ$, because $L$ is connected. On the other hand, we
have a natural injective map $N/\mathrm{Cen}(N)=\mathrm{Inn}(N)\rInto\Aut(N)$.
Under this map, $\Aut(N)^\circ\cong N^\circ/\mathrm{Cen}(N)\cap N^\circ$,
see \cite{iwa} p.~514, Theorem 1$'$ or \cite{HMCompact}
Theorem~9.82.
The subgroup of $\Aut(N)^\circ$ that centralizes $N^\circ$
is therefore isomorphic to $\mathrm{Cen}(N^\circ)/\mathrm{Cen}(N)\cap N^\circ$.
Thus the $L$-action on $N$ is given by a map 
$L\rTo (\mathrm{Cen}(N^\circ)/\mathrm{Cen}(N)\cap N^\circ)^\circ$.
The target group of this map is a quotient of the compact connected abelian group
$\mathrm{Cen}(N^\circ)^\circ$. Since this group is in particular abelian,
we end up with a map 
$L\rTo^{\mathrm{ab}} L/[L,L]\rTo^\psi (\mathrm{Cen}(N^\circ)/\mathrm{Cen}(N)\cap N^\circ)^\circ$.
Now $L$ is simply connected and thus $L/[L,L]\cong\RR^m$, for some $m\geq 0$.
By Lemma \ref{LemmaL1}, there exists a lift $\tilde\psi:L/[L,L]\rTo \mathrm{Cen}(N^\circ)^\circ$.
Now we consider the composite $\beta=j\circ\tilde\psi\circ \mathrm{ab}$,
\[
 L\rTo^{\mathrm{ab}}L/[L,L]\rTo^{\tilde\psi}\mathrm{Cen}(N^\circ)^\circ\rTo^j N.
\]
where $j(x)=x^{-1}$. Then we have
\[
\beta(\ell)^{-1}n\beta(\ell)=\alpha(\ell)(n)
\]
for all $n\in N$ and $\ell\in L$.
Now
\[
 N\times L\rTo N\rtimes L,\qquad\phi(n,\ell)=(n\beta(\ell),\ell)
\]
is the desired isomorphism.
\qed
\end{Lem}
The following result, which is Theorem~C in our introduction, 
is a global version of Glu\v skov's Theorem~A in \cite{Gluskov}.
\begin{Thm}
\label{Theorem1-1} 
Let $G$ be a locally compact group. Then
for every identity neighborhood $U$
there is a compact  subgroup $N$ contained in $U$, a simply connected 
Lie group $L$, and an open and continuous homomorphism 
$\phi:N\times L\rTo G$ with discrete kernel such that $\phi(n,1)=n$
for all $n\in N$.

\proof We divide the proof into several steps.

\medskip\noindent\emph{Claim 1. The result holds if $G$ is connected.}
We apply Iwasawa's Local Splitting Theorem~\ref{IwasawaSplittingTheorem}.
The fact that $C$ is a  local Lie group on an open $n$-cell means that
there is a Lie group $L$, an $n$-cell identity neighborhood
$W\subseteq L$, and a homeomorphism $\gamma: W\rTo C$ for which $x,y,xy\in W$ implies
$\gamma(xy)=\gamma(x)\gamma(y)$. We may assume $L$ to be simply
connected. Then $\gamma$ extends to a unique homomorphism of topological groups
$\gamma: L\rTo G$, see \cite{HMCompact}, Corollary~A2.26 and A2.27. 
Since $C$ is in the
centralizer of $N$, so is the subgroup $\gamma(L)$ generated by $C$.
Hence the map 
\[
\phi: N\times L\rTo G,\qquad\phi(n,\ell)=n\gamma(\ell),
\]
is a continuous homomorphism which maps $N\times W$
homeomorphically  onto the identity neighborhood $NC$ of $G$.
Thus $\ker\phi$ is discrete and $\phi$ is locally open and hence open.
Clearly $\phi(n,1)=n$. The  assertion follows in this special case.

\medskip\noindent\emph{Claim 2. The result holds if $G/G^\circ$ is compact.}
Then every identity neighborhood contains a compact normal subgroup $P$ such 
that $G/P$ is a Lie group, see \cite{montzip} Ch.~4.6,~p.~175. 
Let $U\subseteq G$ be an identity neighborhood.
By Theorem~\ref{IwasawaSplittingTheorem}, 
the identity component $G^\circ$ has a relatively open identity 
neighborhood $QC\cong Q\times C$ with a compact
normal subgroup $Q\unlhd G^\circ$ contained in $U$
and  an open $n$-cell local Lie group $C$. 
We may assume that the $n$-cell $C$ 
contains no subgroup besides $\{1\}$.
Let $\psi: Q\times L\rTo G^\circ$ be the 
surjective homomorphism guaranteed by Claim~1 of the proof,
and put
\[
 \gamma:L\rTo G,\qquad\gamma(\ell)=\psi(1,\ell).
 \]
Let ${\mathcal N}(G)$ be the filter basis
of compact normal subgroups $P\unlhd G$
such that $G/P$ is a Lie group.
Since the filter basis ${\mathcal N}(G)$
converges to $1$, there is a $P\in{\cal N}(G)$ such that $P\subseteq U$ 
and $P\cap G^\circ\subseteq QC$. Since $C$ contains no nontrivial
subgroups, we conclude that $P\cap G^\circ\subseteq Q$
(because $\pr_C(P\cap G^\circ)$ is a subgroup of $C$).
Since $G/P$ is a Lie group and $G/G^\circ$ is compact, 
$G/PG^\circ$ is finite. Thus $PG^\circ$ is open in $G$, and we 
may as well assume that $G=PG^\circ$. 
The group $\gamma(L)$ centralizes $Q$ and normalizes $P\unlhd G$.
Therefore it normalizes the compact group $N=PQ\subseteq G$.
We put
\[
\alpha:L\rTo\Aut(N),\qquad \alpha(\ell)(n)=\gamma(\ell)n\gamma(\ell)^{-1}
\]
Then the semidirect product $N\rtimes_{\alpha} L$ 
has a continuous homomorphism
\[
\phi:N\rtimes_\alpha L\rTo G,\qquad \phi(n,\ell)=n\gamma(\ell).
\]
Its image contains  $P$,
$Q$ and $\gamma(L)$. Therefore it maps onto $PQ\gamma(L)=P G^\circ =G$.
Since $L$ and $N$ are $\sigma$-compact, the group
$N\rtimes_\alpha L$ is $\sigma$-compact. By the Open Mapping Theorem 
for Locally Compact Groups (see e.g.~\cite{HMCompact} p.~669), 
$\phi$ is open. We claim that the kernel is discrete.
Let $W=\gamma^{-1}(C)$. Then $N\times W$ is an identity
neighborhood of $N\rtimes_\alpha L$.
Suppose that $(n,w)\in (N\times W)\cap\ker\phi$.
Then $n=\gamma(w)^{-1}=c\in C$ and $n=qp$ for some $p\in P$ and $q\in Q$.
Thus $p=q^{-1}c\in QC\subseteq G^\circ$. It follows that
$p\in P\cap G^\circ\subseteq Q$. Thus we may assume that $p=1$, and this
implies $n=q=c=1$. Thus $\phi$ has a discrete kernel.
Next we note that $N^\circ\subseteq G^\circ$, and thus
$N^\circ\subseteq Q$. Therefore $L$ centralizes $N^\circ$.
By Lemma~\ref{LemmaL2}, we have an isomorphism 
\[
 N\times L\rTo N\rtimes_\alpha L
\]
which is the identity on $N\times\{1\}$. Thus we have proved Claim~2.

\medskip\noindent\emph{Claim 3. The result holds for arbitrary locally compact
groups $G$.}
In such a group $G$, there is an open subgroup $H\subseteq G$ such that
$H/H^\circ$ is compact. 
By Claim~2 we find a simply connected Lie group
$L$, a compact subgroup $N$ and an open homomorphism
\[
 N\times L\rTo H\rInto G.
\]
\qed
\end{Thm}

\section{The proof of Theorem B}

We now extend our results from Section~\ref{ThmASection} to locally compact groups. 
\begin{Num}\textbf{\!\!\!Locally contractible spaces\ }
\label{PropertiesOfLCSpaces}
By way of comparison we remind the reader that a space $X$ is 
called {\it locally contractible} if it satisfies the following condition.
\begin{description}
 \item[(LC*)] For every point $x\in X$ and every neighborhood $V$ of $x$ there
 exists a neighborhood $U\subseteq V$ of $x$ which is contractible in $V$.
\end{description}
A locally contractible space is locally arcwise connected and piecewise 
contractible, that is,
\begin{center}
 (LC*) $\Longrightarrow$ (PC).
\end{center}
A neighborhood retract of a locally contractible space is 
locally contractible, see \cite{HY}~Theorem 4--42 or \cite{HuRetracts}~I.9.
It follows that 
a product of two spaces is locally contractible if and only if the two factors
are locally contractible. We also note that
being locally contractible is a local property of a space.
\end{Num}
In view of Theorem~\ref{Theorem1-1} the following lemma is the main step in this extension.
\begin{Lem}[The Reduction Lemma]
\label{ReductionLemma}
Let $L$ be a Lie group and let $N$ be a compact
group. Suppose that $K$ is a closed subgroup of the locally compact group
$G=N\times L$. If $X=G/K$ is locally contractible, then $X$ is a manifold.
If $N$ acts faithfully on $X$, then $N$ is a compact Lie group and hence
$G$ is a Lie group.

\proof
Since $N$ is compact and normal in $G$, the group $NK\subseteq G$ is closed,
see \cite{HewittRoss} II 4.4.
Because $N$ is a direct factor in $G$, the group $NK$ splits as
$NK=N\times H$, and $H\subseteq L$ is a closed Lie subgroup.
The natural map 
\begin{diagram}
N\times H &\rTo &N\times L\\
&&\dTo\\
&& (N\times L)/(N\times H)
\end{diagram}
is a locally trivial principal bundle because this is true for the map
$L\rTo L/H$, see
\cite{Warner}, Theorem~3.58. It follows that the associated bundle
\begin{diagram}
(N\times H)/K&\rTo&(N\times L)/K\\
&&\dTo\\
&&(N\times L)/(N\times H)
\end{diagram}
is also locally trivial. 
Since  $(N\times L)/K$ is locally contractible,
the same is true for the fiber $(N\times H)/K$ by the remarks
in \ref{PropertiesOfLCSpaces}. 
Now the compact group $N$ acts transitively on the fiber
\[
F=(N\times H)/K=NK/K\cong N/N\cap K.
\]
By Theorem~\ref{MainTheoremADetailed}, the homogeneous space
$F$ is a closed manifold. Since the base space
$B=(N\times L)/(N\times H)\cong L/H$ is also a manifold and since 
$X=(N\times L)/K$ is locally homeomorphic to $B\times F$,
we have that $X$ is a manifold.
If $N$ acts faithfully on $G/K$, then $N$ acts 
faithfully on $F$, because a subgroup $P\subseteq N\cap K$ which is normal 
in $N$ is also normal in $N\times L$.
Hence $N$ is a compact Lie group in this case.
\qed
\end{Lem}
\begin{Cor}
\label{XIsManifold}
Let $G$ be a locally compact group and let $K\subseteq G$ be a closed subgroup.
If $X=G/K$ is locally contractible, then $X$ is a manifold.

\proof
Let $\phi:N\times L\rTo G$ be as in Theorem~\ref{Theorem1-1}. Then
$H=\phi(N\times L)$ is an open subgroup of $G$. Therefore
$H$ has an open orbit $Y=HK/K\subseteq X$. By Lemma~\ref{ReductionLemma},
this open set $Y$ is a manifold. 
It follows from the homogeneity of $X$ that $X$ itself is a manifold.
\qed
\end{Cor}
The following consequence is immediate. This is Theorem~B in the introduction.
\begin{Thm}
\label{GIsLieGroup}
Let $G$ be a locally compact group and let $K\subseteq G$ be a closed subgroup.
Assume that $X=G/K$ is locally contractible and that $X$ is connected or that
$G/G^\circ$ is compact. If $G$ acts faithfully on $X$, then $G$ is a Lie group.

\proof
Assume first that $X$ is connected. Then
we argue similarly as in the proof of Corollary~\ref{XIsManifold}
and we consider the open subgroup $H=\phi(N\times L)$.
The $H$-orbit of the coset $gK\in G/K$ is the open set $HgK/K$.
Because $X$ is connected, this implies that $H$ acts transitively
on $X$. In particular, $N\subseteq H$ acts faithfully on $H/H\cap K\cong G/K$.
By Lemma~\ref{ReductionLemma}, $H$ is a Lie group. Since $H\subseteq G$
is open, $G$ is also a Lie group.
If $G/G^\circ$ is compact, then $G$ is a Lie group by
\cite{montzip}~Ch.~6.3, Corollary on p.~243.
\qed
\end{Thm}
Again, we restate this result in terms of transformation groups.
\begin{Cor}
Let $G$ be a locally compact group, with $G/G^\circ$ compact. 
If $X$ is a locally contractible, locally compact space and if $G\times X\rTo X$
is a transitive continuous faithful action, then $G$ is a Lie group and 
$X$ is a manifold.

\proof
Let $x\in X$ be a point and let $K\subseteq G$ denote the stabilizer of $x$.
Our assumptions imply that $G$ is $\sigma$-compact, hence 
the natural continuous map $G/K\rTo X$ is a homeomorphism, see for 
example \cite{Stroppel} 10.10.
\qed
\end{Cor}

\section{A historical review}

In his influential  1974 paper \cite{janos} Szenthe stated 
the following result on locally compact groups:

\medskip\noindent
{\em Theorem 4 \cite{janos}. Let a $\sigma$-compact group $G$ with compact $G/G^\circ$
be an effective and transitive topological transformation group of
a locally compact and locally contractible space $X$. Then $G$ is
a Lie group and $X$ is homeomorphic to a coset space of $G$.}

\medskip\noindent
Szenthe's statement provided
a result which was needed and applied in various 
areas, notably in geometry, see for example
\cite{age,ber1,ber2,salz,knarr,knarr2,KramerHabil}.
The proof  of Theorem 4 was based, 
among  other things, on the following statement on compact groups.

\medskip\noindent
{\em Lemma 6 \cite{janos}. Let $G$ be a compact group, $H\subseteq G$ a closed subgroup,
$\chi:G\rTo G/H$ the canonical projection, $A\subseteq G$ a closed invariant
subgroup and $A'=\chi(A)$. If $A'$ is contractible over $G/H$, then $A\subseteq H$.}

\medskip\noindent
However, in 2011, Sergey Antonyan \cite{antii}
found the following simple counterexample to Szenthe's Lemma 6:

\medskip\noindent
\textbf{Example.} Let $G=\SS^1\subseteq\CC^*$, let  $H=\{1\}$ and
and put $A=\{\pm1\}$. Then $A'=A$ is contractible in $G$,
but is not contained in $H$.

\medskip\noindent
It was also noted the mid-nineties by Salzmann and his school \cite{salz}
that Szenthe's method of approximating
a locally compact group $G$ by Lie groups  forces $G$ to be
metric, that is, first countable, see \cite{bickel}. 
Thus even if a  substitute 
method for Szenthe's Lemma~6 for compact groups could be obtained 
by some  correct argument, the ensuing version of Theorem~4
could only be valid for
first countable locally compact groups.

\bigskip

{\raggedright
Karl H. Hofmann \\
Technische Universit\"at Darmstadt \\
Schlossgartenstra{\ss}e 7\\
64289 Darmstadt, Germany\\ 
\makeatletter
{\tt hofmann{@}mathematik.tu-darmstadt.de}

\medskip

Linus Kramer \\
Mathematisches Institut\\
Einsteinstr. 62\\
48149 M\"unster,  Germany\\
\makeatletter
{\tt linus.kramer{@}uni-muenster.de}
}
\newpage
\begin{center}
 \LARGE\bf Erratum to:\\
Transitive actions of locally compact groups\\ on locally contractible
spaces
\end{center}
%
\bigskip
\begin{center}
 Karl H.~Hofmann and Linus Kramer
\end{center}


\bigskip
\bigskip
After the article {\it Transitive actions of locally compact groups on locally contractible
spaces} went to online publication, 
we noticed that the proof of Lemma 3.4 is incorrect. The problem is that
a lift of a homotopically constant map in a fibration need not be homotopically constant.
This issue can be resolved as follows. The numbering and the references are as in our original
article. We have added the references \cite{Whitehead} and \cite{hofmorr3}.

We first prove a preparatory Lemma about coverings of piecewise contractible spaces. 
We call a map $E\rTo^p B$ a \emph{covering map} if every point
$b\in B$ has a neighborhood $V$ which is evenly covered, i.e. $p^{-1}(V)\rTo^p V$
is isomorphic to the product map $V\times F\rTo V$, for some discrete space $F$.
We do not impose (local) connectivity assumptions on $B$ or $E$.

\medskip\noindent\textbf{Lemma 3.1$\frac12$}
{\em Suppose that $E\rTo^p B$ is a covering map. If $B$ is piecewise contractible, then
$E$ is piecewise contractible.}

\proof
(See \cite{hofmorr3}, Lemma 10.77.)
Let $e\in E$ be a point, and let $U$ be an open neighborhood of $p(e)$ which is evenly
covered. 
Replacing $U$ by a smaller neighborhood of $p(e)$ if necessary, we may assume that 
$U$ is contractible in $B$ by a homotopy $h:U\times[0,1]\rTo B$. 
Let $s:U\rTo E$ be a cross section of $p$ over $U$ such that $s(U)$
is a neighborhood of $e$. A covering map is automatically a fibration
\cite[Theorem I.7.12]{Whitehead}, hence there exists
a lift $\tilde h:U\times[0,1]\rTo E$ of $h$ with $\tilde h_0=s$. Then $\tilde h_1$ maps
$U$ into the discrete fiber $F=p^{-1}(h_1(p(e)))
$. 
The preimage $V$ of $\tilde h_1(e)$ is therefore open in $U$.
Thus $s(V)$ is an open neighborhood of $e$ which can be contracted in $E$.
\qed

\medskip\noindent\textbf{Lemma 3.4}
{\em Let $G,K,X,M$ be as in Lemma~3.3. If $G$ acts faithfully
on $X$, then $M$ is a compact Lie group.}

\proof (See \cite{hofmorr3}, Proof of Lemma 10.78.)
We put $L=M\cap K$ and we identify $X$ with the quotient space $M/L$.
It follows from Theorem~2.7
that there is a compact connected semisimple Lie group
$S$, a compact connected finite dimensional abelian group $A$, and a surjective homomorphism
$q:A\times S\rTo M$ with a totally disconnected kernel $E$. We may assume that $q$ maps
$A$ isomorphically onto $\mathrm{Cen}(M)^\circ$. Then $E$ is finite and $q$ is a covering 
homomorphism.
We note that $L$ intersects $q(A)=\mathrm{Cen}(M)^\circ$ trivially, because the action is faithful. Hence
$L$ injects into the compact semisimple Lie group $M/q(A)$. In particular, $L$
is a compact Lie group. Since $q$ is a covering homomorphism, the group $H=q^{-1}(L)$
is also a compact Lie group. Then the compact connected group 
$(\{1\}\times S)H^\circ=T\times S\subseteq A\times S$ is also a
Lie group by Lemma~2.3, and therefore $T$ is a finite dimensional torus.
It follows from [19, Theorem 8.78(ii)] that $A$ splits as 
$A\cong T\times B$, for some compact abelian group $B$.
Now we have $(A\times S)/H^\circ\cong B\times((T\times S)/H^\circ)$.
By Lemma 3.1$\frac12$, the space $(A\times S)/H^\circ$ is piecewise contractible,
and by Lemma 3.1, $B$ is piecewise contractible. In particular, the path component of
the identity in $B$ is open in $B$. Since $B$ is connected, it is therefore path connected.
Since $B$ has finite dimension, it is metrizable [21, Theorem 8.49] 
and therefore a finite dimensional torus [21, Theorem 8.46(iii)]. 
Thus $A$ itself is a finite dimensional torus, and $A\times S$ is a compact Lie group. 
Then $M=q(A\times S)$ is also a Lie group.
\qed

\bigskip

\raggedright
Karl H. Hofmann \\
Technische Universit\"at Darmstadt \\
Schlossgartenstra{\ss}e 7\\
64289 Darmstadt, Germany\\ 
\makeatletter
{\tt hofmann{@}mathematik.tu-darmstadt.de}

\medskip

Linus Kramer \\
Mathematisches Institut\\
Einsteinstr. 62\\
48149 M\"unster,  Germany\\
\makeatletter
{\tt linus.kramer{@}uni-muenster.de}

\end{document}